\documentclass[14pt]{article}
\usepackage[a4paper, left=1in, right=1in]{geometry}
\usepackage[utf8]{inputenc}

\usepackage{ifthen}
\usepackage{xargs}

\usepackage{amsmath,amsfonts,amssymb,amsthm,indentfirst,enumerate,textcomp}
\usepackage[english]{babel}
\usepackage{indentfirst, array}
\usepackage{amscd,latexsym}
\usepackage{mathrsfs}
\usepackage{tabularx}
\usepackage{multirow}
\usepackage[noadjust]{cite}

\usepackage{graphicx}
\usepackage{textcase}
\usepackage{marginnote}
\usepackage{color}
\usepackage{multirow}
\usepackage{bigdelim}

\theoremstyle{plain}
\newtheorem{statement}{Statement}

\newtheorem{lemma}{Lemma}
\newtheorem{theorem}{Theorem}

\DeclareMathOperator{\perm}{perm}

\newcommand{\calP}{\mathcal{P}}
\newcommandx{\PP}[2][1=, 2=]{%
	\ifthenelse{\equal{#1}{}}{\calP}{\calP_{#1}}%
	\ifthenelse{\equal{#2}{}}{}{\val{#2}}%
	}	
\newcommandx{\PI}[2][1=, 2=]{%
	\ifthenelse{\equal{#1}{}}{\calP^{*}}{\calP^{*}_{#1}}%
	\ifthenelse{\equal{#2}{}}{}{\val{#2}}%
	}
	
\newcommandx{\D}[1][1=]{%
	\ifthenelse{\equal{#1}{}}{\Delta}{\Delta_{#1}}%
	}

\newcommand{\J}{\mathcal{J}}
\newcommand{\K}{\mathcal{K}}

\newcommand{\abs}[1]{\left|#1\right|}
\newcommand{\val}[1]{\left(#1\right)}
\newcommand{\curly}[1]{\left\{#1\right\}}

\newcommand{\nrm}[1]{\overline{#1}}

\newcommand{\Q}{\mathbb{Q}}
\newcommand{\C}{\mathbb{C}}
\newcommand{\Z}{\mathbb{Z}}
\newcommand{\N}{\mathbb{N}}
\newcommand{\R}{\mathbb{R}}

\newcommand{\al}{\alpha}
\newcommand{\ga}{\gamma}
\newcommand{\la}{\lambda}

\newcommandx{\A}[2][1=, 2=]{%
	\ifthenelse{\equal{#1}{}}{\mathbb{A}}{\mathbb{A}_{#1}}%
	\ifthenelse{\equal{#2}{}}{}{\val{#2}}%
	}

\title{Counting real algebraic numbers with bounded derivative of minimal polynomial}
\author{A. Kudin, D. Vasilyev}
\date{}

\begin{document} 
\maketitle
\begin{abstract}
	In this paper we consider the problem of counting algebraic numbers $\al$ of fixed degree $n$ and 
	bounded height $Q$ such that the derivative of the minimal polynomial $P_{\al}(x)$ of $\al$ is bounded, 
	$\abs{P'_{\al}(\al)} < Q^{1-v}$.
	This problem has many applications to the problems of the metric theory of Diophantine approximation.
	We prove that the number of $\al$ defined above on the interval 
	$\val{-\frac12, \frac12}$ doesn't exceed $c_1(n)Q^{n+1-\frac{1}{7}v}$ for $Q>Q_0(n)$ and $1.4 \le v \le \frac{7}{16}(n+1)$.
	Our result is based on an improvement to the lemma on the order of zero approximation by irreducible divisors 
	of integer polynomials from A. Gelfond's monograph "Transcendental and algebraic numbers".
	The improvement provides a stronger estimate for the absolute value of the divisor in real points which are located far enough 
	from all algebraic numbers of bounded degree and height and it's based on
	the representation of the resultant of two polynomials as the determinant of Sylvester
	matrix for the shifted polynomials.
\end{abstract}

\section{Introduction}
Many problems in the metric theory of Diophantine approximation and the theory of transcendental numbers are formulated 
in terms of real, complex or $p$-adic number sets satisfying the following inequalities:
\begin{align}
	\label{mahler}
	\left| P(x) \right|      < H(P)^{-w_1}, \quad
	\left| P(z) \right|      < H(P)^{-w_2}, \quad
	\left| P(\omega) \right|_p < H(P)^{-w_3},
\end{align}
where $w_i > 0$, $x \in \R$, $z \in \C$, $\omega \in \Q_p$,
for infinitely many polynomials $P(x)$ from some class $\calP \subset \Z[x]$.
Here and throughout the paper for a polynomial $P(x) = a_n x^n + ... + a_1 x + a_0 \in \C[x]$
we denote by $H(P)$ its "naive height", i.e. $H(P)=\max_{0 \le i \le n} |a_i|$.
Complexity of the sets defined above motivates the search for their best possible approximations by combinations of simpler sets
(real intervals, complex circles or $p$-adic cylinders).

For simplicity let's consider only the real case now. 
Solutions to the first inequality of \eqref{mahler} are located (see \cite[Part I, Chapter I, §2, Lemma 2]{Spr67}) in the intervals of the form 
\begin{align*}
    |x-\alpha_1|<2^{n-1}|P(x)||P'(\alpha_1)|^{-1},
\end{align*}
where $\alpha_1$ is the closest to $x$ root of $P(x)$. These intervals can be quite large for small values of $|P'(\alpha_1)|$.
A natural solution to this problem is to find an upper bound for the number of polynomials having a small derivative at a root.

This approach has been used in R. Baker's work \cite{RBaker76}, for example.
For some integer $n \ge 1$ and real $H \ge 1$, $v \ge 0$ he considers the set $\tilde{\calP}_n(H,v)$ 
of primitive irreducible polynomials $P(x)$ of degree $n$ and height $H$, 
which are leading (that is, $\abs{a_n}=H$),
such that there exists a root $\alpha_1 \in \C$ of $P(x)$ with $|P'(\alpha_1)|<H^{1-v}$ 
and also some additional limitations specific to the problem being solved are implied.
R. Baker has proved for $0 \le v < 1$ and $H$ large enough that
\begin{align} \label{Pn(H,v)<}
    \# \tilde{\calP}_n \left(H,v\right) < c_1(n) H^{n-v},
\end{align}
where $c_1(n)$ is some value that depends on $n$ only.
Using this result he obtained for $n \ge 3$ and $w_1 > \frac{1}{3}\left(n^2+n-3\right)$ the exact upper bound for the Hausdorff dimension
of $x \in \R$, for which there are infinitely many integer polynomials of degree $n$ satisfying the first inequality of \eqref{mahler}.
The problem of calculating the Hausdorff dimension of this set was completely solved by V. Bernik \cite{Bern83} using a different approach.
But nevertheless estimates similar to \eqref{Pn(H,v)<} can be useful in many problems of the metric theory of Diophantine approximation, for example \cite{Spr67,BBG10-comp},
and they are interesting on their own as generalizations of the problems related to the distribution of algebraic numbers and algebraic integers
\cite{Spr67, Cas57, Sch80, Evertse04, BBG10-comp, BM2010, BBG2016, Bernik2015}.

Let's introduce some useful notation.
In the paper $\mu_k(A)$ will denote the Lebesgue measure of a measurable set $A \subset \R^k$, $k \in \N$.
Let $\PP[\le n]$ (resp. $\PP[=n]$) 
be the set of integer polynomials $P \in \Z[x]$ with $\deg P \le n$ (resp. $\deg P = n$) and let 
$\PP[\le n][Q]$ (resp. $\PP[=n][Q]$) 
be the set of polynomials $P \in \PP[\le n]$ (resp. $P \in \PP[=n]$) with $H(P) \le Q$.
In the paper Vinogradov symbol will be used extensively.
We will write $f \ll_{x_1,x_2,...} g$ if there is a real value $c > 0$,
which depends on $x_1, x_2, ...$, but doesn't depend on $f$ and $g$, such that $f \le c g$,
and also $f \asymp_{x_1,x_2,...} g$ means that both $f \ll_{x_1,x_2,...} g$ and $f \gg_{x_1,x_2,...} g$ are true.
Sometimes we write the hidden Vinogradov symbol value $c(x_1, x_2, ...)$ explicitly 
while slightly abusing the notation by using the same symbol $c(x_1, x_2, ...)$ for actually different values throughout the paper.
For a matrix $M \in \R^{m \times n} = (a_{ij})$ define its permanent by 
\begin{gather*}
	\perm M = \left\{ 
	\begin{array}{l}
		\underset{\sigma \in \mathrm{P}(n,m)}{\sum}	a_{1 \sigma(1)} \cdot ... \cdot a_{m \sigma(m)}, \mbox{  if  } m \le n,\\
		\perm M^{T}, \mbox{  if  } m > n,
	\end{array}
	\right.
\end{gather*}
where $\mathrm{P}(n,m)$ is the set of all $m$-permutations of $\{1,...,n\}$.

For some set $D \subseteq \C$, integer $n \ge 1$ and 
real numbers $Q \ge 1$, $v \ge 0$
denote by $\PP[n][Q,v,D]$ 
the set of primitive irreducible polynomials $P \in \PP[=n][Q]$
with a positive leading coefficient,
having a root $\alpha \in D$ such that 
\begin{align} \label{P'<}
	\abs{P'\val{\alpha}} < Q^{1-v}.
\end{align}
By definition, the set $\PP[n][Q,v,D]$ contains only polynomials which are minimal for some algebraic numbers,
therefore, by counting the elements of $\PP[n][Q,v,D]$ 
we essentially count algebraic numbers with certain properties.

In this paper we consider only algebraic numbers in the interval $I_0 = \val{-\frac{1}{2},\frac{1}{2}}$.
Previously an upper bound similar to \eqref{Pn(H,v)<} was obtained for a slightly wider range of $v$.

\begin{theorem}[{\cite{LimDerAtRootUpper32}}] \label{th:upper32}
	For $n \ge 1$
	there exist $c_1(n) > 0$ and $Q_0(n)>0$ such that
	for any $Q > Q_0(n)$ and for all $0 \le v \le \frac{3}{2}$ we have
	\begin{align*}
	        \# \PP[n][Q,v,I_0] \le c_1(n) Q^{n+1-v}.
	    \end{align*}
\end{theorem}
\noindent Also a lower bound was obtained.
For technical reasons we need to replace condition \eqref{P'<} with the following:
\begin{align} \label{P'<chi}
	\abs{P'\val{\alpha}} < C_D Q^{1-v},
\end{align}
for some $C_D > 0$. Denote by $\PP[n][Q,v,D,C_D]$ the set of polynomials similar to $\PP[n][Q,v,D]$,
but having the derivative values determined by \eqref{P'<chi} instead of \eqref{P'<}.
\begin{theorem}[{\cite{LimDerAtRootLower}}] \label{th:lower}
	For $n \ge 2$
	there exist $c_1(n) > 0$, $Q_0(n)>0$ and $C_D(n) > 0$ such that
	for any $Q > Q_0(n)$ and for all $0 \le v \le \frac{n+1}{3}$ we have	
	\begin{align*}
		\# \PP[n][Q,v,D,C_D] \ge c_1(n) Q^{n+1-2v}.
	\end{align*}
\end{theorem}
In Theorem \ref{th:upper32} the range of $v$ doesn't depend on $n$, 
which significantly limits the applications for large values of $n$.
We prove the following upper bound, thus partially addressing this issue.
\begin{theorem} \label{th:main}
	For $n \ge 3$
	there exist $c_1(n) > 0$ and $Q_0(n)>0$ such that 
	for any $Q > Q_0(n)$ and for all $1.4 \le v \le \frac{7}{16}(n+1)$
	the following estimate holds:	
	\begin{equation} \label{th:main:stat}
		\# \PP[n][Q,v,I_0] \le c_1(n) Q^{n+1-v\ga},
	\end{equation}				
	with $\ga = \frac{1}{7}$.
\end{theorem}
In the core of the method we use in the proof
is an improvement to one lemma from
A. Gelfond's monograph "Transcendental and algebraic numbers".
\begin{lemma}[{\cite[Chapter 3, §4, Lemma VI]{Gelfond}}]
	Let $P \in \PP[\le m][Q]$, $m \in \N$, $Q \ge 1$, be a primitive polynomial.
	If we have $\abs{P(\xi)} < Q^{-\tau}$, $\tau>6m$, in some transcendental point $\xi \in \R$,
	then there exists a divisor $t_1(x)$ of $P(x)$, which is a power of some irreducible integer polynomial,
	such that for all $Q>Q_0(m,\xi)$ we have
	\begin{align*}
		\abs{t_1(\xi)} < Q^{-\tau+6m}.
	\end{align*}
\end{lemma}
\noindent
This lemma has been further improved by relaxing the condition on $\tau$ and obtaining stronger estimates for the absolute value of the divisor.
\begin{lemma}[{\cite[Lemma 14]{Bern83}}]
	Let $I \subset (-m,m)$, $m \in \N$, be an interval and
	let $P \in \PP[\le m][Q^{\la}]$, $\la \ge 0$, $Q \ge 1$, be a polynomial.
	If we have $\abs{P(\xi)} < Q^{-\tau}$, $\tau>3m\la$, for any point $\xi \in I$,
	then there exists a divisor $t_1(x)$ of $P(x)$, which is a power of some irreducible integer polynomial,
	such that for all $Q>Q_0(m)$ we have
	\begin{align*}
		\abs{t_1(\xi)} \ll_m Q^{-\tau+m\la} \,\, \forall \, \xi \in I.
	\end{align*}
\end{lemma}
\begin{lemma}[{\cite{KudinGelfond}}]
	Let $I \subset (-m,m)$, $m \in \N$, be an interval and
	let $P \in \PP[\le m][Q^{\la}]$, $\la \ge 0$, $Q \ge 1$, be a polynomial.
	If we have $\abs{P(\xi)} < Q^{-\tau}$, $\tau > 0$, for any point $\xi \in I$,
	then there exists a divisor $t_1(x)$ of $P(x)$, $\deg t_1=m_1$, $H(t_1)=Q^{\la_1}$,
	which is a power of some irreducible integer polynomial,
	such that for any $\delta > 0$ and for all $Q>Q_0(m,M,\delta)$ we have
	\begin{align*}
		\abs{t_1(\xi)} < Q^{- \tau + m\la - m_1\la_1 - (m-m_1)(\la-\la_1) + \delta}  \,\, \forall \, \xi \in I.
	\end{align*}
\end{lemma}
\noindent
In certain points $\xi \in \R$ which are located far enough from all algebraic numbers of bounded degree and height
using a corollary of Lemma \ref{lm:TI2000}
we can obtain a stronger estimate:
	\begin{align*}
		\abs{t_1(\xi)} < Q^{- \tau + \frac12m\la - \frac12 m_1\la_1 - \frac12 (m-m_1)(\la-\la_1) + \delta},
	\end{align*}
as indicated by Statements \ref{th:main:statDIV} and \ref{th:main:statDIVI} below.
We expect that the method may be further improved, allowing us to increase the value of $\ga$ up to $1$ in \eqref{th:main:stat}.

\section{Auxiliary statements}

The following lemmas will be used in the proof of Theorem \ref{th:main}.
\begin{lemma}[{\cite[Lemma 10]{Bern83}}] \label{lm:P(B)<}
	Let $I \subset \R$ be an interval and let $B$ be some measurable subset of $I$, $\mu_1 B \gg_n \mu_1 I$.
	If for some polynomial $P \in \PP[\le n]$ we have $\abs{P(\xi)} < L$ for any point $\xi \in B$,
	then $\abs{P(\xi)} \ll_n L$ holds for any $\xi \in I$.
\end{lemma}

\begin{lemma}[{\cite[Chapter 1, §2, Lemma IV]{Gelfond}}] \label{lm:H1H2=H12}
	For any $n \in \N$ 
	there exist real values $C_A(n)>0$ and $C_B(n)>0$, such that
	for any non-zero polynomial $P(x) \in \C[x]$, $\deg P \le n$,
	which is a product of $k \in \N$ polynomials $P_i(x) \in \C[x]$, we have	
	\begin{gather*}
		C_A(n) H(P) \le H(P_1) \cdot ... \cdot H(P_k) \le C_B(n) H(P).
	\end{gather*}
\end{lemma}

\begin{lemma}[{\cite[Lemma 3.3]{TI2000}}] \label{lm:TI2000}
    Let $P_1(x), P_2(x) \in \C[x]$, $\deg P_i = n_i \ge 0$, be polynomials and 
    let $\xi \in \C$. 
    The resultant of polynomials $P_1$ and $P_2$ is equal to 
    the determinant of Sylvester matrix for shifted polynomials 
    $S_1(x)=P_1(x+\xi)$ and $S_2(x)=P_2(x+\xi)$:
    \begin{align} \label{lm:TI2000:sylvester}
    \arraycolsep=0.8pt
    R(P_1, P_2)=
    \begin{array}{c c c c c c c c c c c c}
	\ldelim|{8}{1mm}&P_1(\xi)&P_1^{'}(\xi)&        &\cdots      &\phantom{ }&  &\frac{1}{n_1!} P_1^{(n_1)}(\xi)&      &                               &\rdelim|{8}{1mm}&\rdelim\}{4}{1mm}[$n_2$]\\
                    &        &\ddots      &\ddots  &            &\phantom{ }&  &                               &\ddots&                               &                &\\
                    &        &            &P_1(\xi)&P_1^{'}(\xi)&\phantom{ }&  &\cdots                         &      &\frac{1}{n_1!} P_1^{(n_1)}(\xi)&                &\\
	                &P_2(\xi)&P_2^{'}(\xi)&        &\cdots      &\phantom{ }&  &\frac{1}{n_2!} P_2^{(n_2)}(\xi)&      &                               &                &\rdelim\}{4}{1mm}[$n_1$]\\
                    &        &\ddots      &\ddots  &            &\phantom{ }&  &                               &\ddots&                               &                &\\
                    &	     &            &P_2(\xi)&P_2^{'}(\xi)&\phantom{ }&  &\cdots                         &      &\frac{1}{n_2!} P_2^{(n_2)}(\xi)&                &
    \end{array}
    \end{align}
\end{lemma}

\begin{lemma} \label{lm:s<s}
	Let $P \in \PP[\le n]$ be a polynomial and
	let $\xi \in \C$ be a point, such that $\abs{\xi - \al} > L > 0$ for any root $\al \in \C$ of $P(x)$.
	We then have
	\begin{gather*}
		\abs{P^{(j)}(\xi)} \ll_n \abs{P(\xi)} L^{-j}, \,\, j = 1,2,... \,.
	\end{gather*}
\end{lemma}
\begin{proof}
	Let $P(x) = a_k (x - \al_1) \cdot ... \cdot (x - \al_k)$, $k \le n$.
	It's not hard to see that
	\begin{gather*}
		\abs{P^{(j)}(\xi)} \le 
			\sum_{1 \le i_1 < i_2 < ... < i_{j} \le k} 
				\frac{ j! \abs{P(\xi)} }
				{ \abs{\xi - \al_{i_1}} \cdot ... \cdot \abs{\xi - \al_{i_{j}}} }
				\ll_n \abs{P(\xi)} L^{-j}, \mbox{ for } j = 1, ..., k, \\
			\abs{P^{(j)}(\xi)} = 0 \ll_n \abs{P(\xi)} L^{-j}, \mbox{ for } j = k+1, ... \, .
	\end{gather*}
\end{proof}

\noindent 
Using Lemma \ref{lm:TI2000} we can prove Lemma \ref{lm:sylvester},
which will be extensively used throughout the paper.
Note that \eqref{lm:sylvester:colsk} with the one-column permanent is essentially
A. Gelfond's lemma \cite[Chapter 3, §4, Lemma V]{Gelfond} for integer polynomials,
and \eqref{lm:sylvester:cols3interval} is an extension of 
V. Bernik's lemma \cite[Lemma 12]{Bern83} 
for polynomials of different degrees and heights.
\begin{lemma} \label{lm:sylvester}
	Let $n_1, n_2 \ge 1$ be integers, such that $n_1 + n_2 \le n$ for some integer $n$, 
	and let $\la_1 \ge 0, \la_2 \ge 0, Q \ge 1$ be reals.
    Let $P_1(x) \in \PP[= n_1][Q^{\la_1}]$, $P_2(x) \in \PP[=n_2][Q^{\la_2}]$
    be integer polynomials, having no common roots, and $\xi \in \val{-\frac12,\frac12}$. 
    
    For any natural $k \le n_1 + n_2$ we have
    \begin{align} \label{lm:sylvester:colsk}
    	\setlength\extrarowheight{5pt}
    	1 \ll_n \perm \val{
    	\begin{array}{c c c}
    	\abs{\nrm{P}^{\phantom{(k)}}_1(\xi)}	&	\cdots	&	\abs{\nrm{P}^{(k)}_1(\xi)}	\\
    											&	\ddots	&	\vdots					\\
    											&			&	\abs{\nrm{P}^{\phantom{(k)}}_1(\xi)}	\\
    	\abs{\nrm{P}^{\phantom{(k)}}_2(\xi)}	&	\cdots	&	\abs{\nrm{P}^{(k)}_2(\xi)}	\\
    											&	\ddots	&	\vdots					\\
    											&			&	\abs{\nrm{P}^{\phantom{(k)}}_2(\xi)}
    	\end{array}
    	}
    	Q^{n_1 \la_2 + n_2 \la_1},
    \end{align}
    where $	\nrm{P}^{(j)}_i(x) = P^{(j)}_i(x) Q^{-\la_i}, \,\, j \ge 0, \,\, i=1,2$.
  
    If $3 \le n_1 + n_2$ and in addition
    \begin{align} \label{lm:sylvester:cols3interval-cond}
    	\abs{P_i(\xi)} \le Q^{-\tau_i}, \, \tau_i \in \R, \, i=1,2,
    \end{align}
    for all points $\xi \in I$ of some interval $I \subseteq \val{-\frac12,\frac12}$, $\mu_1 I=Q^{-\eta}$, $\eta > 0$,
    then for any $\delta > 0$ and for all $Q>Q_0(n,\delta)$ we have 
    \begin{align} \label{lm:sylvester:cols3interval}
    	3\min\curly{\tau_1 + \la_1, \tau_2 + \la_2} - 2 \eta < n_1 \la_2 + n_2 \la_1 + \delta.
    \end{align}
\end{lemma}
\begin{proof}
	From \eqref{lm:TI2000:sylvester} it follows that 
	\begin{align} \label{lm:sylvester:T}
		1 \le \abs{R(P_1,P_2)} = \abs{\det T} Q^{n_1 \la_2 + n_2 \la_1},
	\end{align}
	where $T$ is a matrix similar to \eqref{lm:TI2000:sylvester}, 
	but having terms $P^{(j)}_i(\xi)$ replaced with $\nrm{P}^{(j)}_i(\xi)$.
	We can estimate the elements of $T$ as follows:
	\begin{align*}
		\abs{\frac{1}{j!} \nrm{P}^{(j)}_i(\xi)} \le c(n), \, j \ge 0, \,\, i=1,2.
	\end{align*}

	For any $k \le n_1 + n_2$ 
	we may consider the expansion of $\det T$ by the first $k$ columns 
	(as it has at least $k$ columns)
	and observe that the absolute value of each summand of the expansion contains as a factor some summand
	of the expansion of the permanent from \eqref{lm:sylvester:colsk} 
	and the absolute values of the other factors can be estimated by $c(n)$.
	There are no more than $c(n)$ different summands in the expansion of $\det T$.
	Therefore, the absolute value of $\det T$ can be estimated from above by the value of the permanent from 
	\eqref{lm:sylvester:colsk} times $c(n)$ and so \eqref{lm:sylvester:colsk} follows immediately from 
	\eqref{lm:sylvester:T}. Note, that if $\min\curly{n_1,n_2}<k$, some summands of the permanent expansion 
	do not necessarily correspond to a summand of the expansion of $\det T$,
	but \eqref{lm:sylvester:colsk} is still true, 
	as the expansion of the permanent only serves as an upper bound.

	Assume that \eqref{lm:sylvester:cols3interval-cond} holds.
	Polynomial $P_1(x)P_2(x)$ has no more than $2n$ roots, so
	we can find a point $\xi_0 \in I$ such that $\abs{\xi_0-\al} \gg_n Q^{-\eta}$
	for each root $\al \in \C$ of $P_1(x)P_2(x)$.
	According to Lemma \ref{lm:s<s} the following estimates are true
	at $\xi_0$:
	\begin{align*}
		\abs{{P}'_i(\xi_0)}  &\ll_n \abs{{P}_i(\xi_0)} Q^{\eta}, \,  i=1,2,\\
		\abs{{P}''_i(\xi_0)} &\ll_n \abs{{P}_i(\xi_0)} Q^{2\eta}, \, i=1,2,
	\end{align*}
	therefore, each summand of the expansion of the three-column permanent from \eqref{lm:sylvester:colsk}
	can be estimated by $c(n) \max \curly{\abs{\nrm{P}_i(\xi_0)}}^{3} Q^{2\eta}$, so we obtain
	\begin{align*}
		1 \le c(n) \max \curly{\abs{\nrm{P}_i(\xi_0)}}^{3} Q^{2\eta + n_1 \la_2 + n_2 \la_1} \le 
			Q^{-3\min\curly{\tau_1 + \la_1, \tau_2 + \la_2}} Q^{2\eta + n_1 \la_2 + n_2 \la_1}.
	\end{align*}
	Estimating $c(n)$ by $Q^{\delta}$ and taking logarithms base $Q$, 
	we obtain \eqref{lm:sylvester:cols3interval}.	
\end{proof}

\section{Proof of Theorem 3}
We assume that the opposite to \eqref{th:main:stat} holds, 
i.e. there exists $n \ge 3$ such that for any $c_1(n) > 0$
there are infinitely many pairs $(Q,v)$ with $Q \rightarrow \infty$ and 
\begin{equation} \label{th:main:<v<}
	1.4 \le v \le \frac{7}{16}(n+1),
\end{equation}
such that
\begin{equation} \label{th:main:PP>c_1}
	\#\PP[n][Q,v,I_0] > c_1(n) Q^{n+1-v\ga},	
\end{equation}
and obtain a contradiction from this assumption for $Q>Q_0(n)$.
For the rest of the proof we fix one such pair $(Q,v)$.

Let $\J$ be a minimal set of non-intersecting half-open intervals $J \subset \R$ of length $Q^{-v}$ 
covering the interval $I_0$.
We have $C_L(n) Q^v \le \#\J \le C_U(n) Q^v$ 
for $Q>Q_0(n)$ and some values $C_L(n) > 0$ and $C_U(n) > 0$.
Obviously,
\begin{gather*}
	\PP[n][Q,v,I_0] \subseteq \bigcup_{J \in \J} \PP[n][Q,v,J],\\
	\#\PP[n][Q,v,I_0] \le \sum_{J \in \J} \#\PP[n][Q,v,J].
\end{gather*}
According to the pigeonhole principle there is an interval $J \in \J$ such that
\begin{equation*}
	\#\PP[n][Q,v,J] \gg_n Q^{n+1-v(1+\ga)}.
\end{equation*}	
In fact, we can deduce more information from \eqref{th:main:PP>c_1} in Statement \ref{th:main:statRho} below.

We fix some small enough value $\D(n) > 0$,
which will control the "precision" of our estimates.
For example, $c(n) < Q^{\D(n)}$ for $Q>Q_0(n)$ for any particular value $c(n)$.
Also we can estimate values $C_A(n)$ and $C_B(n)$ from Lemma \ref{lm:H1H2=H12} as follows:
\begin{align} \label{lm:H1H2=H12:Q>}
	Q^{-\frac{\D}{n}} \le C_A(n), C_B(n) \le Q^{\frac{\D}{n}}.
\end{align}
As we can see later, it's sufficient to take 
\begin{align} \label{th:main:D}
	\D(n) = \frac{1}{128 n}.
\end{align}

\begin{statement} \label{th:main:statRho}
	For any $c_2(n) > 0$ we can choose $c_1(n) > 0$ such that for any $Q>Q_0(n)$
	there exist a real $\rho$, $0 \le \rho \le v$, and 
	a set of intervals $\K \subseteq \J$,
	\begin{gather}
		\label{th:main:statRho:muK}
		\#\K \ge Q^{v-\rho-\D},\,\,
		\mu_1\val{\bigcup_{K \in \K}K} \ge Q^{-\rho-\D},
	\end{gather}
	such that for any interval $K \in \K$ we have
	\begin{gather} \label{th:main:statRho:countP}
		\#\PP[n][Q,v,K] \ge c_2(n) Q^{n+1-v(1+\ga)+\rho}.
	\end{gather}
\end{statement}

\begin{proof}
	Let $T$ be an integer such that $T \D \le v < (T+1) \D$.
	By definition of $v$ and $\D$, $0 \le T \le T_0(n) = \frac{n}{\D}$.
	Take $A = n+1-v(1+\ga)$ and define the following subsets of $\J$:
	\begin{align*}
		\K_{-1}   &= \curly{J \in \J | \, \#\PP[n][Q,v,J] < c_2(n) Q^{A}}, 			\\
		\K_t      &= \curly{J \in \J | \, c_2(n) Q^{A+t\D} \le \#\PP[n][Q,v,J] < c_2(n) Q^{A+(t+1)\D}}, \,\,\,	t = 0,...,T-1, \nonumber \\
		\K_{T}    &= \curly{J \in \J | \, c_2(n) Q^{A+T\D} \le \#\PP[n][Q,v,J]}.
	\end{align*}
	Obviously, $\J = \K_{-1} \sqcup \K_0 \sqcup ... \sqcup K_{T}$.	
	Suppose that $\#\K_t < Q^{v-(t+1)\D}$ for $t=0,...,T$.
	Trivially, $\#\K_{-1} \le C_U(n) Q^v$ and $\#\K_{T} = 0$.	
	We then have
	\begin{gather*}
		c_1(n) Q^{n+1-v\ga} < \#\PP[n][Q,v,I_0] \le \sum_{J \in \J} \#\PP[n][Q,v,J] 
		= \sum_{t=-1}^{T} \sum_{J\in\K_{t}} \#\PP[n][Q,v,J] < \\
		< c_2(n) Q^{n+1-v(1+\ga)} C_U(n) Q^v + \sum_{t=0}^{T-1} c_2(n) Q^{n+1-v(1+\ga)+(t+1)\D} Q^{v-(t+1)\D} \le \\
		\le c_2(n) Q^{n+1-v\ga} \val{C_U(n) + T} \le c_2(n) Q^{n+1-v\ga} \val{C_U(n) + \frac{n}{\D}}.
	\end{gather*}
	By choosing $c_1(n) = c_2(n)\val{C_U(n) + \frac{n}{\D}}$ 
	we obtain a contradiction.
	Therefore, $\#\K_t \ge Q^{v-(t+1)\D}$ for some $t=0,...,T$.
	Now taking $\rho = t\D$ and $\K=\K_{t}$ finishes the proof.
\end{proof}

Among the intervals $K \in \K$ we choose a special interval $K_0$, 
at least half of the points of which are located far enough from all algebraic numbers of bounded degree and height.
With the help of this fact, 
knowing the absolute values of certain polynomials we may estimate the absolute values of their derivatives.
\begin{statement} \label{th:main:statB0}
For any $Q>Q_0(n)$
there exist an interval $K_0 \in \K$ and a measurable set $B_0 \subseteq K_0$ with the following properties:
\begin{enumerate}
	\item 
	the measure of $B_0$ is at least
	\begin{gather} \label{th:main:statB0:muB_0}
		\mu_1 B_0 \ge \frac{1}{2}\mu_1 K_0,
	\end{gather}
	\item
	for any integer polynomial $s(x)$, such that
	\begin{gather} \label{th:main:statB0:poly-any:def}
		\deg s = m \le n, H(s) = Q^{\la} \le Q^{1+\D},
	\end{gather}	
	and for any $\xi \in B_0$ we have
	\begin{gather*} \label{th:main:statB0:poly-any:prop}
		\abs{s'\val{\xi}} < \abs{s \val{\xi}} Q^{\rho + \la(m+1) + 3\D},
	\end{gather*}	
\end{enumerate}
\end{statement}

\begin{proof}
	Let $A_1$ be the set of all algebraic numbers of degree not exceeding $n$ and of height not exceeding $Q^{1+2\D}$.
	For each $\al \in A_1$ define a real interval (note that $\al$ is not necessarily real):
	\begin{align} 
		\label{th:main:statB0:sigma}
		\sigma(\al) &= \curly{x \in \R \Big| \abs{x-\al} < H(\al)^{-\deg \al-1}Q^{-\rho-2\D}},
	\end{align}
	and let $S = \bigcup_{\al \in A_1} \sigma(\al)$ be the union of the intervals defined above.
	\noindent
	For $Q>Q_0(n)$ we have 	
	\begin{gather*} \label{th:main:statB0:muS1}
			\mu_1 S  
			\le 
				\sum_{1 \le m \le n} \,\,
				\sum_{1 \le H \le Q^{1+2\D}} \,\,
				\sum_{\substack{\al \in A_1: \\ \deg \alpha = m, \\ H(\alpha)=H}} 
				\mu_1 \sigma(\alpha)
			\le
				\sum_{1 \le m \le n} \,\,
				\sum_{\,\,1 \le H \le Q^{1+2\D}} 
				c(m) H^{m} H^{-m-1} Q^{-\rho-2\D}
			\le \\
			\le
				c(n)
				Q^{-\rho-2\D}
				\sum_{1 \le m \le n} \,\,
				\sum_{\,\,1 \le H \le Q^{1+2\D}} 
				H^{-1} \le c(n) Q^{-\rho-2\D} \ln \val{Q} \le \frac{1}{2} Q^{-\rho-\D},
	\end{gather*}
	therefore, from \eqref{th:main:statRho:muK} it follows that 
	there is an interval $K_0 \in \K$ such that 
	$\mu_1\val{K_0 \bigcap S} \le \frac{1}{2} \mu_1 K_0$.
	If we let $B_0 = K_0 \bigcap \val{\R \setminus S}$, we can clearly see that \eqref{th:main:statB0:muB_0} holds.		
	Any polynomial $s(x)$ defined by \eqref{th:main:statB0:poly-any:def} 
	may have only roots $\al$ such that $\al \in A_1$, $\deg \al \le \deg s$, $H(\al) \ll_n H(s)$.
	Therefore, by \eqref{th:main:statB0:sigma} for any $\xi \in B_0$ and any root $\al$ of $s(x)$ we obtain 
	\begin{gather*}
		\abs{\xi - \al} \ge H(\al)^{-\deg \al -1}Q^{-\rho-2\D} \ge c(n) H(s)^{-\deg s-1} Q^{-\rho-2\D},
	\end{gather*}
	which, according to Lemma \ref{lm:s<s}, gives us for $Q>Q_0(n)$ the following:
	\begin{gather*}
		\abs{s'(\xi)} \le c(n) \abs{s(\xi)} Q^{\rho+\la(m+1)+2\D} < \abs{s(\xi)}Q^{\rho+\la(m+1)+3\D}.
	\end{gather*}
\end{proof}

Points of $B_0$ are special in some sense.
In particular, given an integer polynomial $R(x)$ small enough at some point $\xi \in B_0$,
Statements \ref{th:main:statDIV} and \ref{th:main:statDIVI} guarantee that one particular integer polynomial divisor of $R(x)$
is substantially smaller at $\xi$ than the other divisors.
\begin{statement} \label{th:main:statDIV}
	Let $S(x) \in \PP[\le m][Q^\la]$, $2 \le m \le n$, $0 \le \la$, be an integer polynomial,
	which is a product of two non-constant integer polynomials $s_1(x)$ and $s_2(x)$, having no common roots,
	such that in some point $\xi \in B_0$ we have:
	\begin{gather}
		\label{th:main:statDIV:R=t1t2}
		S(x) = s_1(x) s_2(x), \\
		\label{th:main:statDIV:mili}
		\deg s_i = m_i \ge 1, H(s_i) = Q^{\la_i} \le Q^{1+\D}, \\
		\nonumber
		\abs{S(\xi)} = Q^{-\tau}, \, \abs{s_i(\xi)} = Q^{-\tau_i}, \, i=1,2,\\
		\label{th:main:statDIV:t1<t2}
		\frac{\abs{s_1(\xi)}}{H(s_1)} = Q^{-\tau_1-\la_1} \le Q^{-\tau_2-\la_2} = \frac{\abs{s_2(\xi)}}{H(s_2)}.
	\end{gather}
	If in addition
	\begin{gather*}
		\tau > \rho + m\la + 5\D,
	\end{gather*}
	then for $Q>Q_0(n)$ the following is true:
	\begin{gather}
		\label{th:main:statDIV:deg}
		m_1 \ge 2, \\
		\label{th:main:statDIV:tau2<}
		\tau_2 \le \tau_2 + \la_2 < \frac{1}{2}\val{m_1\la_2 + m_2\la_1 + \D} < 
			\frac{1}{2}m\la - \frac{1}{2}m_1\la_1 - \frac{1}{2}(m-m_1)(\la-\la_1) + \D,\\
		\label{th:main:statDIV:tau1>}
		\tau_1 > \tau - \frac{1}{2}m\la + \frac{1}{2}m_1\la_1 + \frac{1}{2}(m-m_1)(\la-\la_1) -\D.
	\end{gather}
\end{statement}
\begin{proof}
	Since $S(x)$ is a product of $s_1(x)$ and $s_2(x)$, by Lemma \ref{lm:H1H2=H12} we have
	\begin{gather*}
		m_1 + m_2 \le m, \, \la_1 + \la_2 \le \la + \frac{\D}{n}, \, \tau_1 + \tau_2 = \tau,
	\end{gather*}
	therefore, we obtain
	\begin{gather*}
		m_1\la_2+m_2\la_1 = (m_1+m_2)(\la_1+\la_2)-m_1\la_1-m_2\la_2 \le \\
		m\la-m_1\la_1-m_2\la_2 + \D, \\
		m_1\la_2+m_2\la_1 \le m_1\val{\la-\la_1+\frac{\D}{n}}+(m-m_1)\la_1 \le \\
		\le m\la - m_1\la_1 - (m-m_1)(\la-\la_1) + \D.
	\end{gather*}
	Due to $\eqref{th:main:statDIV:mili}$ we may apply Statement \ref{th:main:statB0} and estimate $\abs{s'_i(\xi)}$ as follows:
	\begin{gather*}
		\abs{s'_i(\xi)} < \abs{s_i(\xi)} Q^{\rho + (m_i+1)\la_i + 3\D}, \, i=1,2.
	\end{gather*}
	As $m_1 + m_2 \ge 2$, we have \eqref{lm:sylvester:colsk} with the two-column permanent:
	\begin{gather*} \label{th:main:statDIV:max}
		1 \ll_n \max\val{
			\abs{\nrm{s}^{\phantom{0}}_1(\xi)}^2,
			\abs{\nrm{s}^{\phantom{0}}_1(\xi)} \abs{\nrm{s}^{\phantom{0}}_2(\xi)},
			\abs{\nrm{s}^{\phantom{0}}_2(\xi)}^2,
			\abs{\nrm{s}^{\phantom{0}}_1(\xi)} \abs{\nrm{s}'_2(\xi)},
			\abs{\nrm{s}'_1(\xi)} \abs{\nrm{s}^{\phantom{0}}_2(\xi)}
			} 
		Q^{m_1\la_2+m_2\la_1},
	\end{gather*}
	where $	\nrm{s}^{(j)}_i(x) = s^{(j)}_i(x) Q^{-\la_i}, \,\, j \ge 0, \,\, i=1,2$.
	If the maximum is attained at the term 
	$\abs{\nrm{s}^{\phantom{0}}_1(\xi)} \abs{\nrm{s}^{\phantom{0}}_2(\xi)}$, a contradiction follows immediately:
	\begin{gather*}
		1 \le c(n) \abs{\nrm{s}^{\phantom{0}}_1(\xi)} \abs{\nrm{s}^{\phantom{0}}_2(\xi)} Q^{m_1\la_2+m_2\la_1} \le
		c(n) Q^{-\rho - m\la - 5\D} Q^{-\la_1 -\la_2} Q^{m \la -m_1\la_1 -m_2\la_2 +\D} \le c(n) Q^{-4\D}.
	\end{gather*}
	We also obtain a contradiction if the maximum is attained at the term 
	$\abs{\nrm{s}'_1(\xi)} \abs{\nrm{s}^{\phantom{0}}_2(\xi)}$ or 
	$\abs{\nrm{s}^{\phantom{0}}_1(\xi)} \abs{\nrm{s}'_2(\xi)}$.
	Let's consider $\abs{\nrm{s}^{\phantom{0}}_1(\xi)} \abs{\nrm{s}'_2(\xi)}$, for example. 
	The same argument works for the other one.
	\begin{gather*}
		1 \le c(n) \abs{\nrm{s}^{\phantom{0}}_1(\xi)} \abs{\nrm{s}'_2(\xi)} Q^{m_1\la_2+m_2\la_1} \le \\
		\le c(n) \abs{s^{\phantom{0}}_1(\xi)} \abs{s^{\phantom{0}}_2(\xi)} Q^{-\la_1-\la_2} Q^{\rho + (m_2+1)\la_2 + 3\D} Q^{m_1\la_2+m_2\la_1} \le \\
		\le c(n) Q^{-\rho - m\la - 5\D} Q^{-\la_1-\la_2} Q^{\rho + (m_2+1)\la_2 + 3\D} Q^{m\la - m_1\la_1 - m_2\la_2 + \D} \le \\
		\le c(n) Q^{-(m_1+1)\la_1-\D}.
	\end{gather*}
	Therefore, the maximum is attained at one of the terms 
	$\abs{\nrm{s}^{\phantom{0}}_1(\xi)}^2$ and $\abs{\nrm{s}^{\phantom{0}}_2(\xi)}^2$,
	and according to \eqref{th:main:statDIV:t1<t2} the maximum is attained at the latter one.
	In other words,
	\begin{gather*}
		1 \le c(n) \abs{\nrm{s}^{\phantom{0}}_2(\xi)}^2 Q^{m_1\la_2+m_2\la_1} <
		Q^{-2\tau_2} Q^{-2\la_2} Q^{m_1\la_2+m_2\la_1 + \D} \le \\
		\le Q^{-2\tau_2} Q^{-2\la_2} Q^{m\la - m_1\la_1 - (m-m_1)(\la-\la_1) + 2\D},
	\end{gather*}				
	which gives us \eqref{th:main:statDIV:tau2<}. 
	Then \eqref{th:main:statDIV:tau1>} immediately follows from \eqref{th:main:statDIV:R=t1t2}.
	
	If $m_1=1$, we can see from \eqref{lm:TI2000:sylvester} that $\abs{\nrm{s}^{\phantom{0}}_2(\xi)}^2$ 
	can be omitted in the maximum, therefore,
	\begin{gather*}
		1 \le c(n) \abs{\nrm{s}^{\phantom{0}}_1(\xi)}^2 Q^{m_1\la_2+m_2\la_1}, \\
		\tau_1 \le \tau_1 + \la_1 < \frac{1}{2}m\la -\frac{1}{2}m_1\la_1 -\frac{1}{2}(m-m_1)(\la-\la_1) +\D \le \frac{1}{2}m\la + \D, \\
		\tau_2 + \la_2 \ge \tau_2 > \tau - \frac{1}{2}m\la - \D \ge \rho + \frac{1}{2}m\la + 4 \D > \frac{1}{2}m\la + \D > \tau_1 + \la_1,
	\end{gather*}
	which contradicts \eqref{th:main:statDIV:t1<t2}.
	Therefore, \eqref{th:main:statDIV:deg} is true.
\end{proof}

\begin{statement} \label{th:main:statDIVI}
	Let $R(x) \in \PP[\le m][Q^\la]$, $2 \le m \le n$, $0 \le \la \le 1$, be a primitive integer polynomial,
	which is a product of $k$, $2 \le k \le n$, powers $t_i(x)$ of different primitive irreducible integer polynomials $p_i(x)$.
	\begin{gather*}
		R(x) = t_1(x) \cdot ... \cdot t_k(x),\\
		t_i(x) = p_i(x)^{e_i}, \, e_i \in \N, \, i=1,...,k.
	\end{gather*}
	If at some point $\xi \in B_0$ we have
	\begin{gather*}
		\abs{R(\xi)} = Q^{-\tau},
		\tau > \rho + m \la + 7\D,
	\end{gather*}
	then for one of the factors $t_d(x) = t_i(x)$ the following holds for any $Q>Q_0(n)$:
	\begin{gather}		
		\nonumber
		\deg t_d = m_d \ge 2, \, H(t_d) = Q^{\la_d}, \, \abs{t_d(\xi)} = Q^{-\tau_d},\\
		\label{th:main:statDIVI:t1>}
		\tau_d > \tau - \frac{1}{2} m\la + \frac{1}{2} m_d\la_d +\frac{1}{2}(m-m_d)(\la-\la_d) - \D.
	\end{gather}
\end{statement}
\begin{proof}
	According to \eqref{lm:H1H2=H12:Q>}, for any integer polynomial divisor $S(x)$ of $R(x)$ we have 
	$H(S) \le H(R) Q^{\frac{\D}{n}} \le Q^{1 + \D}$, 
	so condition \eqref{th:main:statDIV:mili} of Statement \ref{th:main:statDIV} is always satisfied for the divisors of $R(x)$.
	
	By definition, $R(x)$ is a product of two primitive integer polynomials
	having no common roots and we may apply Statement \ref{th:main:statDIV}:
	\begin{gather}	
		\label{th:main:statDIVI:R=s1s2}
		R(x) = s_1(x) s_2(x), \\
		\nonumber
		\deg s_i = m_i, \, H(s_i) = Q^{\la_i}, \abs{s_i(\xi)} = Q^{-\tau_i}, \, i=1,2, \\
		\label{th:main:statDIVI:R=s1l1<s2l2}
		\frac{\abs{s_1(\xi)}}{H(s_1)} = Q^{-\tau_1-\la_1} 
		\le Q^{-\tau_2-\la_2} = \frac{\abs{s_2(\xi)}}{H(s_2)},
	\end{gather}	
	thus obtaining the following:
	{
	\setlength{\belowdisplayskip}{0pt}
	\begin{gather}	
		\label{th:main:statDIVI:t2l2<}
		\tau_2 + \la_2 < \frac{1}{2}m\la - \frac{1}{2}m_1\la_1 - \frac{1}{2}(m-m_1)(\la-\la_1) + \D,
	\end{gather}	
	}
	{
	\setlength{\abovedisplayskip}{0pt}
	\begin{align}	
		\begin{aligned}
		\label{th:main:statDIVI:t(s1)>}
		\tau_1 > \tau - \frac{1}{2}m\la + \frac{1}{2}m_1\la_1 + & \frac{1}{2}(m-m_1)(\la-\la_1) -\D > \\
		> \rho + \frac{1}{2}m\la + \frac{1}{2}m_1\la_1 + &6\D >
		\rho + m_1\la_1 + 5\D.
		\end{aligned}
	\end{align}		
	}
	If $s_1(x)=t_i(x)$ for some $i$, Statement \ref{th:main:statDIVI} is proved.	
	Otherwise, we may again write down $s_1(x)$ 
	as a product of two primitive integer polynomials having no common roots
	and apply Statement \ref{th:main:statDIV} to $s_1(x)$ (as we have \eqref{th:main:statDIVI:t(s1)>}):
	\begin{gather}		
		\nonumber
		s_1(x) = s_{11}(x) s_{12}(x), \\
		\nonumber
		\deg s_{1i} = m_{1i}, \, H(s_{1i}) = Q^{\la_{1i}}, \abs{s_{1i}(\xi)} = Q^{-\tau_{1i}}, \, i=1,2, \\
		\frac{\abs{s_{11}(\xi)}}{H(s_{11})} = Q^{-\tau_{11}-\la_{11}} 
		\nonumber
		\le Q^{-\tau_{12}-\la_{12}} = \frac{\abs{s_{12}(\xi)}}{H(s_{12})},
	\end{gather}	
	which gives us:
	\begin{gather}	
		\label{th:main:statDIVI:t12l12<}
		\tau_{12} + \la_{12} < \frac{1}{2}m_1\la_1 - \frac{1}{2}m_{11}\la_{11} - \frac{1}{2}(m_1-m_{11})(\la_1-\la_{11}) + \D.
	\end{gather}
	From \eqref{th:main:statDIVI:t2l2<} and \eqref{th:main:statDIVI:t12l12<} we obtain:
	\begin{gather*}
		\tau_{12} + \tau_2 \le (\tau_{12} + \la_{12}) + (\tau_2 + \la_2)< \frac{1}{2}m\la + 2\D, \\
		\tau_{11} \ge \tau - \tau_{12} - \tau_2 > \tau - \frac{1}{2}m\la - 2\D > \frac{1}{2}m\la + 5\D,
	\end{gather*}
	therefore,
	\begin{gather*}
		\frac{\abs{s_{11}(\xi)}}{H(s_{11})} < Q^{-\frac{1}{2}m\la - 5\D} < Q^{-\frac{1}{2}m\la - 3\D} < \frac{\abs{s_{12}(\xi)s_{2}(\xi)}}{H(s_{12})H(s_{2})}Q^{-\D} < \frac{\abs{s_{12}(\xi)s_{2}(\xi)}}{H(s_{12}s_{2})},
	\end{gather*}
	so we're back to \eqref{th:main:statDIVI:R=s1s2} and \eqref{th:main:statDIVI:R=s1l1<s2l2}
	with $R(x)=\val{s_{11}(x)}\val{s_{12}(x)s_2(x)}$, 	
	but the degree of the first factor is strictly less than initially.
	After doing a finite amount of such steps the first factor will necessarily be some $t_i(x)$,
	and so we prove Statement \ref{th:main:statDIVI}.
\end{proof}

We have an interval $K_0$ such that $\eqref{th:main:statRho:countP}$ holds, 
so by the pigeonhole principle we can find two polynomials $P_1$ and $P_2$ with the major $n-1$ coefficients
being close.
Subtracting them, we obtain new polynomials $R_m(x)$ for different values of $m$.
\begin{statement} \label{th:main:statRm}	
	Let $m_0$ be an integer such that $v(1+\ga)-\rho-1 \le m_0 < v(1+\ga)-\rho$.
	We may further assume that $v(1+\ga) > \rho + 2$, so $m_0 \ge 2$.
	If we choose $c_2(n)$ large enough, then for $Q>Q_0(n)$ and for any integer $m_0 \le m \le n$ 
	there exists a primitive integer polynomial $R_m(x)$ satisfying the following:
	\begin{align} \label{th:main:statRm:estR}
		\begin{split}
			&2 \le \deg R_m \le m,      \\
			&H\val{R_m} \le Q^{\la}, \la = \frac{v(1+\ga)-2-\rho}{m-1}, \\
			&\abs{R_m(\xi)} < Q^{1-2v+\D} \,\, \forall \xi \in K_0,\\
			&\abs{R'_m(\xi)} < Q^{1-v+\D} \,\, \forall \xi \in K_0.
		\end{split}
	\end{align}	
	Each polynomial $R_m(x)$ has a divisor $t_d(x)$, which is a power of some primitive irreducible integer polynomial,
	such that
	\begin{gather}		
		\nonumber
		\deg t_d = m_d \ge 2, \, H(t_d) = Q^{\la_d}, \, \abs{t_d(\xi)} < Q^{-\tau_d} \,\, \forall \xi \in K_0,\\
		\label{th:main:statRm:t1>}
		\tau_d = (2v-1) - \frac{1}{2} m\la + \frac{1}{2} m_d\la_d +\frac{1}{2}(m-m_d)(\la-\la_d) - 3\D,\\
		\label{th:main:statRm:t1>simple}
		\tau_d > \frac{1}{2}\val{\rho + v(3-\ga) + m_d\la_d - 1} - 3\D.
	\end{gather}
\end{statement}
\begin{proof}
	By definition, any $P(x) = a_0 + a_1 x + ... + a_n x^n \in \PP[n][Q,v,K_0]$ has a root $\al \in K_0$.
	Using the Taylor series expansion with Lagrange remainder, we obtain for any $\xi \in K_0$
	\begin{gather}
		\label{th:main:statRm:P<}
		\abs{P(\xi)} \le \abs{P'(\al)} \abs{\xi-\al} + \frac{1}{2} \abs{P''(\xi_1)} \abs{\xi-\al}^2 \le c(n) Q^{1-2v} < Q^{1-2v+\D}, \\
		\label{th:main:statRm:P'<}
		\abs{P'(\xi)} \le \abs{P'(\al)}+\abs{P''(\xi_1)} \abs{\xi-\al} \le c(n) Q^{1-v} < Q^{1-v+\D},
	\end{gather}
	where $\xi_1 \in (\min(\xi,\al), \max(\xi,\al))$, and so $\abs{P''(\xi_1)} \ll_n Q$.
		
	Coefficients $a_2, ..., a_n$ of such polynomials $P(x)$ are located in the intervals $L_i = [-Q,Q]$, $i=2,...,n$.
	Take some integer $1 \le m \le n$ and real $0 \le \la \le 1$ and cover each interval $L_i$ with a minimal set of non-intersecting half-open intervals 
	$M_{j_i} \subset \R$ of length $\frac{1}{n^2}Q^{\la}$ for $i=2,...,m$ and 
	of length $Q^0=1$ for $i=m+1,...,n$.
	We have $\#\{M_{j_i}\} \le 4 n^2 Q^{1-\la}$ for $i=2,...,m$ and 
	$\#\{M_{j_i}\} \le 4Q$ for $i=m+1,...,n$.
	If we let $c_2(n) = (2n)^{2n}$ then for $m$ and $\la$ such that
	\begin{gather*}
		\la(m-1) \ge v(1+\ga)-2-\rho,
	\end{gather*}
	we have
	\begin{gather*}
		\#\PP[n][Q,v,K_0] \ge c_2(n) Q^{n+1-v(1+\ga)+\rho} > \\
		> (2n)^{2(n-1)} Q^{(m-1)(1-\la)+n-m} \ge \#\curly{M_{j_2} \times ... \times M_{j_n}},
	\end{gather*}
	so there are at least two different polynomials $P_1(x)$, $P_2(x) \in \PP[n][Q,v,K_0]$,
	having coefficients $a_2, ..., a_n$ in the same parallelepiped $M_{j_2} \times ... \times M_{j_n}$.
	Let $R_m(x) = P_2(x)-P_1(x) = r_0+r_1 x + ... \,$.
	Obviously, $\deg R_m \le m$ and $\abs{r_i} \le \frac{1}{n^2}Q^{\la}$ for $i=2,...,m$.
	Estimates \eqref{th:main:statRm:P<} and \eqref{th:main:statRm:P'<} still hold for $R_m(x)$ and $R'_m(x)$,
	so we can estimate $\abs{r_1}$ assuming $Q>Q_0(n)$:
	\begin{gather*}
		\abs{m r_m \xi^{m-1}  + ... + 2 r_2 \xi + r_1} = \abs{R_m'(\xi)} \ll_n Q^{1-v} \,\, \forall \xi \in K_0,\\
		\abs{r_1} \le c(n)Q^{1-v}  + \frac{3}{4}Q^{\la} \le \frac{1}{4} + \frac{3}{4}Q^{\la} \le Q^{\la},
	\end{gather*}
	and then estimate $\abs{r_0}$, thus estimating $H(R_m)$:
	\begin{gather*}
		\abs{m r_m \xi^{m}  + ... + r_1 \xi + r_0} = \abs{R_m(\xi)} \ll_n Q^{1-2v} \,\, \forall \xi \in K_0,\\
		\abs{r_0} \le c(n)Q^{1-2v} + \frac{3}{4}Q^{\la} \le \frac{1}{4} + \frac{3}{4}Q^{\la}  \le Q^{\la},\\
		H(R_m) \le Q^{\la}.
	\end{gather*}
	If $\deg R_m \le 1$ then, using the estimates above, we can prove that 
	\begin{align*}
		\abs{r_1} &\le c(n)Q^{1-v} < 1,\\
		\abs{r_0} &\le c(n)Q^{1-2v} < 1,
	\end{align*}
	therefore, $R_m(x) = 0$, 
	which contradicts $P_1(x) \not= P_2(x)$.
	If $v(1+\ga) \le \rho + 2$, we may choose $m=1$, which gives a contradiction by the same argument.	
	If $R_m(x)$ is not primitive, we may cancel out the common factor of its coefficients
	without affecting the validity of the estimates \eqref{th:main:statRm:estR}.
	
	If $R_m(x)$ is a power of some primitive irreducible polynomial, taking $t_d(x)=R_m(x)$ finishes the proof.
	Otherwise, observe that Statement \ref{th:main:statDIVI} holds at each point $\xi \in B_0 $:
	\begin{gather*}
		\abs{R_m(\xi)} < Q^{-\tau}, \tau = 2v-1-\D,
	\end{gather*}
	and also
	\begin{gather} \label{th:main:statRm:ml<}
		m\la = m \frac{v(1+\ga) - 2 - \rho}{m-1} \le m \frac{v(1+\ga) - 1 - \rho}{m} = v(1+\ga) - 1 - \rho,
	\end{gather}
	therefore,
	\begin{gather*} 
		\tau = 2v-1-\D = \rho + (v(1+\ga) -1 -\rho) + (v(1-\ga)-\D) \ge \\
		\ge \rho + m\la + v(1-\ga)-\D > \rho + m\la + 7\D,
	\end{gather*}
	as long as we satisfy 
	\begin{gather*}
		v(1-\ga) > 8\D(n),
	\end{gather*}
	which is true according to \eqref{th:main:<v<} and \eqref{th:main:D}.
	Therefore, for each point $\xi \in B_0$ we may extract a divisor $t_{d,\xi}(x)$ of $R_m(x)$, 
	for which \eqref{th:main:statDIVI:t1>} holds.
	As the number of such divisors of $R_m(x)$ doesn't exceed $n$, 
	for at least one divisor $t_d(x)$ estimate \eqref{th:main:statDIVI:t1>} holds for all points $\xi \in B_d$ 
	of some large enough subset $B_d$ of $B_0$, $\mu_1 B_d \ge \frac{1}{n} \mu_1 B_0 \gg_n \mu_1 K_0$.
	Therefore, according to Lemma \ref{lm:P(B)<}, we may extend estimate \eqref{th:main:statDIVI:t1>} for $t_d(x)$ from $B_d$ to the whole interval $K_0$
	with loss of $\D$, i.e. we obtain \eqref{th:main:statRm:t1>}.
	Substituting the upper bound \eqref{th:main:statRm:ml<} for $m\la$, we obtain \eqref{th:main:statRm:t1>simple}.
\end{proof}

\begin{statement} \label{th:main:statp1}
	Divisors $t_d(x)$ of all polynomials $R_m(x)$, $m_0 \le m \le n$, produced by Statement
	\ref{th:main:statRm}, are the powers of the same primitive irreducible polynomial.
	Let $t_e(x)$ be such power of minimal degree.
	Then $t_e(x)$ satisfies the following constraints:
	\begin{align} \label{th:main:statp1:te}
		\begin{split}
		\deg t_e       &= m_e      \le  m_0 < v(1+\ga)-\rho, \\
		H(t_e)         &= Q^{\la_e} <   Q^{\frac{v(1+\ga)-2-\rho}{n-1} + \frac{\D}{n}},\\
		\abs{t_e(\xi)} &< Q^{-\tau_e} \,\, \forall \xi \in K_0,
		\end{split}
	\end{align}
	with $\tau_e$ estimated by \eqref{th:main:statRm:t1>simple}.
\end{statement}
\begin{proof}
	For integer $m_0 \le m \le n-1$ consider polynomials $R_a(x)=R_m(x)$ and $R_b(x)=R_{m+1}(x)$ :
	\begin{gather*}
		\deg R_a \le m_a = m,   \, H\val{R_a} \le Q^{\la_a}, \, \la_a = \frac{v(1+\ga)-2-\rho}{m-1}, \\
		\deg R_b \le m_b = m+1, \, H\val{R_b} \le Q^{\la_b}, \, \la_b = \frac{v(1+\ga)-2-\rho}{m},
	\end{gather*}
	and their respective divisors $t_1(x)$ and $t_2(x)$, 
	\begin{gather} 
		\label{th:main:statp1:degHt1}
		2 \le \deg t_1 = m_1 \le \bar{m}, \,
		H(t_1) = \la_1 < \la_a + \frac{\D}{n} \le \bar{\la}, \,
		\abs{t_1(\xi)} < Q^{-\tau_1}, \, \forall \xi \in K_0,\\
		\label{th:main:statp1:degHt2}
		2 \le \deg t_2 = m_2 \le \bar{m}, \,
		H(t_2) = \la_2 < \la_b + \frac{\D}{n} \le \bar{\la}, \,
		\abs{t_2(\xi)} < Q^{-\tau_2}, \, \forall \xi \in K_0,
	\end{gather}
	with $\tau_1$ and $\tau_2$ estimated by \eqref{th:main:statRm:t1>simple} and
	\begin{gather*}		
		\bar{m} = m+1, \, \bar{\la} = \frac{v(1+\ga)-\rho}{m+1} + \frac{\D}{n}.
	\end{gather*}
	Assume that $t_1(x)$ and $t_2(x)$ are the powers of different primitive irreducible polynomials,
	i.e. they don't have common roots.
	We're going to use symmetric estimates 
	\eqref{th:main:statRm:t1>simple}, \eqref{th:main:statp1:degHt1}, \eqref{th:main:statp1:degHt2}
	for $t_1(x)$ and $t_2(x)$,
	therefore, we may assume w.l.o.g that $\tau_1 + \la_1 \le \tau_2 + \la_2$.
	
	We have $m_1 + m_2 \ge 4$, so we may use \eqref{lm:sylvester:cols3interval} on the interval $K_0$ taking $\delta=\D$:
	\begin{gather*}
		3\tau_1 - 2v \le 3\min\curly{\tau_1 + \la_1, \tau_2 + \la_2} - 2 v < m_1 \la_2 + m_2 \la_1 + \D 
		\le m_1 \bar{\la} + \bar{m} \la_1 + \D = \\
		= \bar{m} \bar{\la} + m_1\la_1 - (\bar{m}-m_1)(\bar{\la}-\la_1) + \D \le \bar{m} \bar{\la} + m_1\la_1 + \D.
	\end{gather*}
	Substituting estimate \eqref{th:main:statRm:t1>simple} for $\tau_1$, we obtain
	\begin{gather*}
		\frac{3}{2}\val{\rho + v(3-\ga) + m_1\la_1 - 1} - 9\D - 2v < v(1+\ga)-\rho + m_1\la_1 + 2\D, \\
		v(3-5\ga) + 5\rho + m_1\la_1 < 3 + 22\D,		
	\end{gather*}
	which makes a contradiction given that
	\begin{gather*}	
		v(3-5\ga) > 3 + 22\D.
	\end{gather*}
	The latter is true according to \eqref{th:main:<v<} and \eqref{th:main:D}.
	
	Therefore, $t_1(x)$ and $t_2(x)$ are necessarily the powers of the same primitive irreducible polynomial, and so
	are all divisors $t_d(x)$ of each $R_m(x)$ for $m = m_0, ..., n$.
	If $t_e(x)$ is the minimal power, then it has to divide both $R_{m_0}(x)$ and $R_n(x)$, 
	therefore, estimates \eqref{th:main:statp1:te} are true.	
\end{proof}

\begin{statement} \label{th:main:statRtP}
	Polynomial $t_e(x)$ from Statement \ref{th:main:statp1} necessarily
	has common roots with each $P(x) \in \PP[n][Q,v,K_0]$.
\end{statement}
\begin{proof}
	Assume that $t_e(x)$ and $P(x)$ do not have common roots.
	We may apply \eqref{lm:sylvester:cols3interval} to $t_e(x)$ and $P(x)$ on the interval $K_0$, as $\deg t_e + \deg P \ge 4$:
	\begin{gather} \label{th:main:statRtP:inequ}
		3\min\val{\tau_e + \la_e, 2v-\D} - 2v < m_e + \la_e n + \D.
	\end{gather}
	Substituting \eqref{th:main:statRm:t1>simple} into the left-hand side of \eqref{th:main:statRtP:inequ} we obtain:
	\begin{gather*}
		\frac{3}{2}\val{\rho + v(3-\ga) + m_e\la_e - 1} - 9\D - 2v < 3\min\val{\tau_e + \la_e, 2v-\D} - 2v.
	\end{gather*}
	The right-hand side of \eqref{th:main:statRtP:inequ} can be estimated using \eqref{th:main:statp1:te} as follows:
	\begin{gather*}
		m_e + \la_e n + \D < v(1+\ga)-\rho + \val{\frac{v(1+\ga)-2-\rho}{n-1} + \frac{\D}{n}}n + \D = \\
		= 2v(1+\ga)-2-2\rho + \frac{v(1+\ga)-2-\rho}{n-1} + 2\D.
	\end{gather*}
	Putting it all together and doing some cancellation, we have
	\begin{gather*}
		v\val{\frac{1}{2} - \frac{7}{2}\ga} + \frac{1}{2} + \frac{7}{2} \rho + \frac{3}{2} m_e\la_e < \frac{v(1+\ga)-2-\rho}{n-1} + 11\D,
	\end{gather*}
	which makes a contradiction if
	\begin{gather*}
		\ga = \frac{1}{7}, \,\, \frac{1}{2} > \frac{v(1+\ga)-2}{n-1} + 11\D.
	\end{gather*}
	The latter condition is satisfied as long as we have
	\begin{gather*}
		11\D(n-1) < 1, \, v \le \frac{n+1}{2(1+\ga)} = \frac{7}{16}(n+1).
	\end{gather*}
	which is true according to \eqref{th:main:<v<} and \eqref{th:main:D}.
\end{proof}

From Statement \ref{th:main:statRtP} it follows that all polynomials $P(x) \in \PP[n][Q,v,K_0]$ have common roots, 
which is not true by definition.
Therefore, assuming the opposite to \eqref{th:main:stat} we obtain a contradiction.
This finishes the proof of Theorem \ref{th:main}.

\section{Acknowledgements}
The authors would like to thank 
Prof. Vasili Bernik for suggesting the problem and the initial methods of the solution, 
Prof. Victor Beresnevich for pointing out the important work of Prof. Kirill Tishchenko and many useful comments,
and Anna Gusakova for reading through the paper and making some important corrections and remarks.

This research was partially supported
by Engineering and Physical Sciences Research Council Programme Grant EP/J018260/1 and 
by Belarusian Republican Foundation for Fundamental Research Grant F17-101.


\newpage
\noindent Alexey Kudin,\\
University of York,\\
Heslington, York YO10 5DD, United Kingdom,\\
e-mail: alexey.kudin@york.ac.uk\\
\\
\noindent Denis Vasilyev,\\
Institute of Mathematics,\\
National Academy of Sciences of Belarus,\\
Surganova str. 11, 220072 Minsk, Belarus,\\
e-mail: vasilyev@im.bas-net.by\\
\\


\begin{thebibliography}{99}
\bibitem{RBaker76}
	R. Baker, \emph{Sprindzuk's theorem and Hausdorff dimension}, Mathematika {23} (1976), no.~2, 184--197.
	
\bibitem{BBG10-comp}
	V. Beresnevich, V. Bernik, F. G{\"o}tze, \emph{The distribution of close conjugate algebraic numbers}, Compos. Math. {146} (2010), no.~5, 1165--1179.

\bibitem{BBG2016}
	V. Beresnevich, V. Bernik, F. Götze, \emph{Integral polynomials with small discriminants and resultants}, Advances in Mathematics {298} (2016), 393--412.

\bibitem{Bern83}
	V. Bernik, \emph{Application of the Hausdorff dimension in the theory of Diophantine approximations}, Acta Arith. {42} (1983), no.~3, 219--253.

\bibitem{Bernik2015}
	V. Bernik, F. G\"otze, \emph{Distribution of real algebraic numbers of arbitrary degree in short intervals}, Izvestiya: Mathematics {79} (2015), no.~1, 21--42.
	
\bibitem{LimDerAtRootUpper32}
	V. Bernik, D. Vasiliev, A. Kudin, \emph{On the number of integral polynomials of given degree and bounded height with small value of derivative at root of polynomial},
	Trudy Instituta matematiki NAN Belarusi {22} (2014), no.~2, 3--8.
	
\bibitem{BM2010}
	Y. Bugeaud, M. Mignotte, \emph{Polynomial root separation}, Int. J. Number Theory. {6} (2010), no.~3, 587--602.

\bibitem{Cas57}
	J.W.S. Cassels, \emph{An introduction to Diophantine approximation}, Cambridge Tracts in Mathematics and Mathematical Physics {45}, Cambridge University Press, New York, 1957.

\bibitem{Evertse04}
	J.-H. Evertse, \emph{Distances between the conjugates of an algebraic number}, Publ.~Math.~Debrecen {65} (2004), 323--340.

\bibitem{Gelfond}
	A.O. Gelfond, \emph{Transcendental and Algebraic Numbers} [In Russian], Moscow, 1952.	
	
\bibitem{LimDerAtRootLower}
	A. Kudin, \emph{Lower bound of the number of integral polynomials of a given degree with a small derivative at the root}, 
	Proc. Nats. Akad. Nauk Belarusi, Ser. Phys. Math {4} (2014), 112--115.

\bibitem{KudinGelfond}
	A. Kudin, \emph{On the order of zero approximation by irreducible divisors of integer polynomials}, 
	Dokl. Nats. Akad. Nauk Belarusi 61(3) (2017), 14--17.
	
\bibitem{Sch80}
	W. Schmidt, \emph{Diophantine approximation}, Lecture Notes in Math. 785, Springer, Berlin, 1980.

\bibitem{Spr67}
	V. Sprindzuk, \emph{Mahler's problem in metric Number Theory}, Math. Monogr. 25, Amer. Math. Soc., Providence, RI, 1969.
	
\bibitem{TI2000}
	K.I. Tishchenko, \emph{On approximation of real numbers by algebraic numbers of bounded degree}, Acta Arith. {94} (2000), no.~1, pp. 1--24.
\end{thebibliography}
\end{document}